\documentclass{amsart}
\usepackage{amsthm}
\usepackage{amsmath}
\usepackage{amssymb}
\usepackage{euscript}
\usepackage{graphicx}

\usepackage{amsfonts}
\usepackage{tikz}
\usepackage{pgfplots}
\usepackage{ragged2e}
\usetikzlibrary{arrows}
\usepackage{amsmath}
\usepackage[english]{babel}
\usepackage{graphicx}

\newtheorem{theorem}{Theorem}
\newtheorem{corollary}[theorem]{Corollary}
\newtheorem{proposition}[theorem]{Proposition}
\newtheorem{lemma}[theorem]{Lemma}
\newtheorem{remark}[theorem]{Remark}
\newtheorem{dfn}[theorem]{Definition}

\newcommand{\C}{\mathcal}

\newcommand{\N}{\mathbb{N}}
\newcommand{\R}{\mathbb{R}}

\newcommand{\w}{\omega}

\newcommand{\sh}{\sigma}

\newcommand{\seq}[1]{\langle #1\rangle}

\title{A characterization of $\omega$-limit sets in subshifts of Baire Space}

\author[J. Meddaugh]{Jonathan Meddaugh}
\address[J. Meddaugh]{Department of Mathematics, Baylor University, Waco, TX 76798--7328,USA}
\email{jonathan\_meddaugh@baylor.edu}

\author[B. E. Raines]{Brian E. Raines}
\address[B. E. Raines]{Department of Mathematics, Baylor University, Waco, TX 76798--7328,USA}
\email{brian\_raines@baylor.edu}


\begin{document}

\begin{abstract}
	In this paper we consider the structure of $\omega$-limit sets in subshifts of Baire space.  We consider both \emph{subshifts of finite type} and \emph{subshifts of bounded type} and we demonstrate that many classical structure theorems for $\omega$-limit sets fail in this context. Nevertheless, we obtain characterizations of $\omega$-limit sets in subshift of finite types and of \emph{attracting} $\omega$-limit sets in subshifts of bounded type.
\end{abstract}

\maketitle

\section{Introduction}

For a continuous map $f:X\to X$ on a metric space and a point $x\in X$, the $\omega$-limit set of $x$ is the set
\[\omega(x)=\bigcap_{N\in\N}\overline{\{f^i(x): i\geq N\}},\]
i.e. the set of accumulation points of the forward orbit of $x$ under the action of $f$. It is immediately apparent that understanding the dynamics of the dynamical system $f:X\to X$ requires an understanding of its $\omega$-limit sets.

The $\omega$-limit sets of a dynamical system are well-studied in the context of compact metric domains $X$. Immediately from the definition, it follows that every $\omega$-limit set is closed (and hence compact) and strongly invariant.  Hirsch, Smith and Zhao demonstrated that every $\omega$-limit set in a compact dynamical system is \emph{internally chain transitive} \cite{Hirsch}. Barwell, Good, Knight and Raines demonstrated that in a number of systems, the converse also holds, i.e. a closed internally chain transitive set is the $\omega$-limit set of some point $x\in X$ \cite{BGKR-omega}. In particular, they demonstrate this fact for subshifts of finite type and for some maps on the interval. This property has also been verified in a few other settings such as in dendritic Julia sets and in a class of `circular' Julia sets, \cite{BMR-ToAppear} and \cite{BR-ToAppear}.  In \cite{MR}, the authors establish that the internal chain transitivity characterization holds in all compact dynamical systems with the \emph{shadowing property} (sometimes referred to as the \emph{pseudo-orbit tracing property}) and this result has been refined by Good and Meddaugh to carry this characterization over to the context of weaker variations of the shadowing property \cite{GoodMeddaugh-ICT}.

For systems $f:X\to X$ with $X$ not compact, less is known about $\omega$-limit sets. Again, it is immediate from the definition that the $\omega$-limit sets are closed and invariant, but may fail to be compact. The problem is additionally complicated by fact that $f$ need not be \emph{uniformly} continous in this setting.

As in \cite{BGKR-omega}, we begin the study of $\omega$-limit sets in the noncompact setting by considering shift spaces over countably infinite alphabets. For a countable set $\Lambda$ endowed with the discrete topology, we consider the dynamical system given by the \emph{shift map} on the space $\Lambda^\omega$ of infinite sequences in $\Lambda$ given by $\sigma\seq{x_n}=\seq{x_{n+1}}$. A subshift of $\Lambda^\omega$ is then a closed, invariant subsystem. As is the case with finite alphabets, subshifts over countable alphabets serve as useful models for a broad class of dynamical systems including countable state Markov systems \cite{Kitchens}.  But, unlike shift spaces over finite alphabets which are equivalent to maps on the Cantor middle-thirds set, these systems are not locally compact anywhere.  In fact they are equivalent to maps on the irrationals in $\mathbb{R}$. 

In the theory of shift spaces over finite alphabets, the \emph{subshifts of finite type} stand out as one of the most well-studied families. They are well understood in terms of their $\omega$-limit sets \cite{BGKR-omega} and their entropy \cite{Lind} among many other dynamical properties. Walters showed that these are precisely the subshifts with the shadowing property \cite{WaltersFiniteType} and Good and Meddaugh have demonstrated that these systems are fundamental to the study of the shadowing property in compact dynamical systems \cite{good-meddaugh}.

In this paper, we discuss subshifts of finite type over countable alphabet and characterize the $\omega$-limit sets in these systems. Additionally, we define subshifts of \emph{bounded} type and demonstrate that these systems are, in a sense, the more natural generalization of subshifts of finite type to the countable alphabet case -- they are precisely those subshifts with the shadowing property. This allows us to, in the spirit of \cite{BGKR-omega}, characterize a subclass of the $\omega$-limit sets (the \emph{attracting} $\omega$-limit sets) in terms of internal chain transitivity.

\section{Preliminaries}

Let $\w=\N\cup\{0\}$.  For any set $\Lambda$, let $\Lambda^\w$ be the set of all infinite words with alphabet $\Lambda$ with topology generated by the product topology taking $\Lambda$ to have the discrete topology.  Let $i,j\in \w$ with $i<j$.  Let $x\in \Lambda^\w$.  Define $x_{[i,j]}=x_ix_{i+1}\dots x_j$, and define $x_{(i,j)}, x_{[i,j)}$, and $x_{(i,j]}$ similarly. For a finite word $w=w_0\cdots w_{n-1}\in\Lambda^n$ and $k\in\omega$, define the \emph{cylinder set of $w_0\cdots w_{n-1}$ centered at $k$} to be the set of all $x\in\Lambda^\omega$ with $x_{[k,k+n)}=w$. It is a standard result that the collection of cylinder sets forms a basis for the topology on $\Lambda^\omega$, and in fact, it is enough to consider only the cylinder sets centered at $0$.

The space $\Lambda^\omega$ is easily seen to be metrizable and it is immediate that the metric given by
\[d(x,y)=\inf\{2^{-n}:x_i=y_i \textrm{ for all } i<n\}\]
is consistent with the topology on $\Lambda^\omega$.

If $\Lambda$ is finite and has at least two elements, then $\Lambda^\omega$ is a compact metric space which is homeomorphic to the Cantor set. In the event that $\Lambda=\omega$, $\Lambda^\omega$ is a non-compact, non-locally compact metric space that is homeomorphic to the irrationals in $\R$.  It is traditionally called \emph{the Baire space}, and it has many uses in descriptive set theory \cite{Descriptive-Set-Theory}.

For a fixed $\Lambda$, we define the \emph{full shift on $\Lambda$} to be the dynamical system $(\Lambda^\w, \sigma)$ where $\sh$ is the \emph{shift map} defined for every $x_0x_1x_2\cdots \in \Lambda^\w$ by $\sh(x_0x_1x_2\cdots)=x_1x_2\cdots$.
 It is easy to check that $\sigma$ is continuous and, in fact, uniformly continuous with respect to the metric $d$.

Let $\Gamma\subseteq \Lambda^\w$.  We call $\Gamma$ a \emph{subshift} of $(\Lambda^\w, \sh)$ provided $\Gamma$ is closed and $\sh$-invariant, i.e. $\sh(\Gamma)=\Gamma$.  Given a subshift $\Gamma$ of $\Lambda^\omega$, let $B_n(\Gamma)$ be the collection of all \emph{allowed words of length $n$}, i.e. $w=w_0\cdots w_{n-1}\in B_n(\Gamma)$ if, and only if, there is some point $x\in \Gamma$ and some $i\in \w$ such that $x_{[i,i+n)}=w_0\cdots w_{n-1}$, in this case we say that \emph{$x$ contains $w$ as a subword} or \emph{$w$ is a subword of $x$}, and if $i=0$ then we call $w$ an \emph{initial segment of $x$}.  Let \[B(\Gamma)=\bigcup_{n\in \N} B_n(\Gamma).\]  We call $B(\Gamma)$ the \emph{set of allowed words} for $\Gamma$.  Notice that, since $\Gamma$ is $\sigma$-invariant, $B(\Gamma)$ is equal to the set of initial segments.

A word $w=w_0\cdots w_{n-1}$ is called \emph{forbidden in $\Gamma$}, or simply \emph{forbidden} when the context is clear, provided $w\notin B(\Gamma)$.  We denote the set of all words forbidden in $\Gamma$ by $F(\Gamma)$, and for each $n\in \N$, we let $F_n(\Gamma)$ denote the words of length $n$ that are forbidden in $\Gamma$.

Let $\Gamma\subseteq \Lambda^\w$ be a subshift. A \emph{basis for the forbidden words of $\Gamma$} is a subset $\C F(\Gamma)$ of $F(\Gamma)$ such that for every $w\in F(\Gamma)$ there is some $w'\in \C F(\Gamma)$ such that $w$ contains $w'$ as a subword.  If $\Gamma$ has a basis for its forbidden words that is finite, we call $\Gamma$ a \emph{subshift of finite type (SFT)}. For $\Lambda$ finite, subshifts of finite type are well-studied \cite{Lind}.

For $\Lambda$ infinite, less is known about subshifts of finite type. In fact, it is not immediately clear that subshifts of finite type with infinite alphabet are the most natural analogue to subshifts of finite type with finite alphabet. In particular, the following notion is another possible analogue.

If $\Gamma$ has a basis for its forbidden words, $\C F(\Gamma)$, such that there is some $n\in \N$ and every element of $\C F(\Gamma)$ has length less than $n$, then we call $\Gamma$ a \emph{subshift of bounded type (SBT)}.  The following observation is immediate.

\begin{proposition}
	Let $\Gamma$ be a subshift of $\Lambda^\omega$. If $\Gamma$ is an SFT, then it is an SBT. If $\Lambda$ is finite and $\Gamma$ is an SBT, then it is an SFT.
\end{proposition}

Subshifts display many interesting dynamical properties. This paper focuses primarily on shadowing properties and their relationship with $\omega$-limit sets.

Let $(X,f)$ be a dynamical system on a metric space. For $\delta>0$, a sequence $\seq{x_i}$ in $X$ is a \emph{$\delta$-pseudo-orbit of $f$} provided that $d(f(x_i),x_{i+1})<\delta$ for all $i\in\omega$. A sequence $\seq{x_i}$ is an \emph{asymptotic pseudo-orbit of $f$} provided that $d(f(x_i),x_{i+1})$ limits to zero.

For a fixed $\epsilon>0$, we say that the sequence $\seq{x_i}$ \emph{$\epsilon$-shadows} $\seq{y_i}$ if $d(x_i,y_i)<\epsilon$ for all $i\in\omega$. Similarly, we say that the sequence $\seq{x_i}$ \emph{asymptotically shadows} $\seq{y_i}$ if $d(x_i,y_i)$ limits to zero.

\begin{dfn}
	Let $(X,f)$ be a dynamical system with $X$ metric. The system $(X,f)$ has the shadowing property provided that for all $\epsilon>0$ there exists $\delta>0$ such that for every $\delta$-pseudo-orbit $\seq{x_i}$ of $f$, there exists a point $z\in X$ such that $\seq{f^i(z)}$ $\epsilon$-shadows $\seq{x_i}$.
\end{dfn}

\begin{dfn}
	Let $(X,f)$ be a dynamical system with $X$ metric. The system $(X,f)$ has the asymptotic shadowing property provided that for every asymptotic pseudo-orbit $\seq{x_i}$ of $f$, there exists a point $z\in X$ such that $\seq{f^i(z)}$ asymptotically shadows $\seq{x_i}$.
\end{dfn}

We will frequently abuse notation and identify the point $z$ with its orbit sequence $\seq{f^i(z)}$ and say that the point $z$ shadows the pseudo-orbit $\seq{x_i}$ in this case.

For an arbitrary sequence $\seq{x_i}$ in a topological space $X$, the $\omega$-limit set of $\seq{x_i}$ is the  following:
\[\omega\seq{x_i}=\bigcap_{N\in\omega}\overline{\{x_i:i\geq N\}}.\]

As with shadowing, in the context of a dynamical system, we often associate a point with its orbit sequence, and thus for a dynamical system $(X,f)$ and $x\in X$, the \emph{$\omega$-limit set of $x$} is the set $\omega(x)=\omega\seq{f^i(x)}$. We say that a subset $Z$ of $X$ is an $\omega$-limit set of the system if there is a point $x\in X$ with $\omega(x)=Z$.

The $\omega$-limit sets of points exhibit a variety of dynamical systems that are not shared by $\omega$-limit sets of arbitrary sequences. Among other things, it is well known that $\omega$-limit sets of points are invariant subsets of the dynamical system. In contrast, $\omega$-limit sets of arbitrary sequences are less well studied, and generally less well-behaved. However, the following results concerning the $\omega$-limit sets of sequences are immediate.

\begin{lemma}\label{omegasubsequences}
	Let $X$ be a metric space. Let $\seq{n_i}$ be a subsequence in $\omega$. Then  $\omega\seq{x_{n_i}}\subseteq\omega\seq{x_i}$.

\end{lemma}

\begin{lemma}\label{omegaequalsomega}
	Let $X$ be a metric space. If $\seq{x_i}$ is asymptotically shadowed by $\seq{z_i}$, then $\omega\seq{x_i}=\omega\seq{z_i}$.
\end{lemma}

An additional concept that we make use of is that of internal chain transitivity. Let $(X,f)$ be a dynamical system with $X$ metric. For $\delta>0$ and $a,b\in X$, a \emph{$\delta$-chain from $a$ to $b$} is a finite sequence $x_0,x_1,\ldots x_n$ such that $d(f(x_i),x_{i+1})<\delta$ for $i<n$ with $a=x_0$ and $b=x_n$.

\begin{dfn}
	Let $(X,f)$ be a dynamical system with $X$ metric. A set $A$ is internally chain transitive (ICT) if for all $a,b\in A$ and all $\delta>0$ there exists a $\delta$-chain from $a$ to $b$ consisting of elements of $A$.
\end{dfn}


\section{Shadowing in subshifts of Baire space}

We begin our exposition of results by exploring the degree to which subshifts of bounded type are the appropriate analogue for subshifts of finite type in the finite alphabet case.

The following characterization of subshifts of bounded type is a generalization of a standard characterization for subshifts of finite type.

\begin{theorem}
	Let $\Gamma$ be a subshift of $\Lambda^\omega$. Then $\Gamma$ is an SBT if and only if there exists $M\in\omega$ such that if $uw,wv\in B(\Gamma)$ and the length of $w$ is at least $M$, then $uwv\in B(\Gamma)$.
\end{theorem}

\begin{proof}
	
	First, suppose that $\Gamma$ is an SBT, we can choose a bounded basis $mathcal F(\Gamma)$ for $F(\Gamma)$. Choose $M$ so that no word in $\mathcal F(\Gamma)$ has length more than $M+1$.
	
	Now, let $u,v,w\in B(\Gamma)$ with $uw,wv\in B(\Gamma)$ and $w$ of length at least $M$. It is then clear that is $p$ is a subword of $uwv$ of length no more than $M+1$, then $p$ is a subword of either $uw$ or $wv$, and hence does not belong to $\mathcal F(\Gamma)$. As such, $uwv$ has no subwords belonging to $\mathcal F(\Gamma)$, and hence $uwv\in B(\Gamma)$ as claimed.
	
	Now, suppose that $\Gamma$ is a subshift of $\Lambda^\omega$ and suppose that $M\in\omega$ satisfies the hypotheses. Let $\mathcal F(\Gamma)$ consist of those forbidden words of length no more than $M+1$. We demonstrate by induction that any forbidden word contains a subword in $\mathcal F(\Gamma)$.
	
	First, we note that if $w$ is a forbidden word of length less than or equal to $M+1$, then $w$ has a subword belonging to $F(\Gamma)$, namely itself. Now, suppose that any forbidden word of length $l$ with $M+1\leq l\leq K$ contains a subword in $\mathcal F(\Gamma)$. Let $p=p_0p_1\cdots p_{K-1}p_K$ be a forbidden word of length $K+1$. Then $w=p_1\cdots d_{K-1}$ is a word of length at least $M$. Thus, at least one of $p_0w$ and $wp_K$ must be forbidden, and by the inductive hypotheses contains a subword belonging to $\mathcal F(\Gamma)$.
	\end{proof}

In \cite{WaltersFiniteType}, it is demonstrated that the shadowing property completely characterizes subshifts of finite type in the setting of a finite alphabet. If we allow for $\Lambda$ to be (possibly) infinite, we have the following analogous results.

\begin{lemma}
	Let $\Gamma$ be an SBT.  Then $(\Gamma, \sh)$ has shadowing.
\end{lemma}

\begin{proof}
	
	Choose $K\in N$ and $\C F(\Gamma)$ a basis of the forbidden words of $\Gamma$ such that all the elements  of $\C F(\Gamma)$ have length less than or equal to $K$. Fix $\epsilon>0$, and choose $M\geq K$ such that if $u,v\in\Gamma$ with $u_{[0,M)}=v_{[0,M)}$, then $d(u,v)<\epsilon$. Now, choose $\delta>0$ such that if $u,v\in\Gamma$ with $d(u,v)<\delta$, then $u_{[0,M)}=v_{[0,M)}$.
	
%
%
%
	
	Let $\seq{x^i}_{i\in \w}$ be a $\delta$-pseudo-orbit in $\Gamma$ for $\sh$ where each $x^i$ is given by $x^i_0x^i_1x^i_2\cdots\in\Gamma$.  Define \[z=z_0z_1z_2\cdots=x^0_0x^1_0x^2_0\cdots\in\Lambda^\omega\] to be the point whose $i$th entry is the initial entry of each $x^i$.  We show that $z\in \Gamma$, and we show that $z$ $\epsilon$-shadows $(x^i)_{i\in \w}$.
	
	Since $(x^i)_{i\in \w}$ is a $\delta$-pseudo-orbit for $\sh$ we see that $d(\sh(x^i),x^{i+1})<\delta$.  So for every $i\in \w$, we have symbol agreement of the following sort: \[x^i_1x^i_2\cdots x^i_{M}=x^{i+1}_0x^{i+1}_1\cdots x^{i+1}_{M-1}.\]  This gives us the following in which every column is made up of identical symbols:\begin{align*}
	x^{i+0}_0 x^{i+0}_1 &x^{i+0}_2 \cdots  x^{i+0}_{M-1} \\
	x^{i+1}_0 &x^{i+1}_1 \cdots  x^{i+1}_{M-2} x^{i+1}_{M-1}\\
	&x^{i+2}_0 \cdots  x^{i+2}_{M-3} x^{i+2}_{M-2} x^{i+2}_{M-1}\\
	\end{align*}
	from which it follows that for all $i\in\omega$ and $0\leq l\leq \min\{M,i\}$
	\[z_i=x_{l}^{i-l}.\]
	In turn, it follows that for all $i\in\omega$ $z_{[i,i+M)}=x^i_{[0,M)}$, and hence $d(\sigma^i(z),x^i)<\epsilon$, i.e. $z$ $\epsilon$-shadows $\seq{x^i}_{i\in\omega}$.
	
	Finally, notice that for each  $i\in\omega$ and $k\leq K$, since  $k\leq K\leq M$, we have $z_{[i,i+k)}=x^i_{[0,k)}\notin \mathcal F(\Gamma)$, and hence $z\in\Gamma$ as desired.
	\end{proof}

\begin{lemma}
	Let $(\Gamma,\sigma)$ have shadowing. Then $\Gamma$ is an SBT.
\end{lemma}

\begin{proof}
	
	Let $\Gamma$ be a subshift such that $(\Gamma,\sigma)$ has shadowing. Fix $\epsilon=1$ and let $\delta>0$ witness shadowing with respect to $\epsilon$. Choose $K\in\N$ such that if $u,v\in\Gamma$ with $u_{[0,K)}=v_{[0,K)}$, then $d(u,v)<\epsilon$. 
	
	Let $l> K$ and suppose that $w=w_0w_1\cdots w_l-1$ is a word of length $l$ and suppose that every subword of $w$ of length less than or equal to $K+1$ is allowed. As such, we can find words $u^i\in\Gamma$ with $u^i_{[0,K+1)}=w_{[i,i+K+1)}$ for $i\leq l-K-1$. For $i\geq l-K-1$, define $u^i=\sigma^{i-(l-K-1)}(u^{l-K-1})$. It is immediate from this construction that $\sigma(u^i)_{[0,K)}=u^{i+1}_{[0,K)}$ and so $\seq{u^i}_{i\in\omega}$ is a $\delta$-pseudo-orbit, and since $(\Gamma,\sigma)$ has shadowing, there exists $z\in\Gamma$ that $\epsilon$ shadows it. Since $\epsilon=1$, it must be the case that $z_i=u^i_0$ for all $i\in\omega$, and hence $z_{[0,l)}=w$ is an allowed word for $\Gamma$.
	
	It follows that if $w$ is a forbidden word of $\Gamma$, then $w$ must contain a forbidden word of length less than or equal to $K+1$, i.e. $F(\Gamma)$ has a bounded basis.
	\end{proof}

Combining these two results yields the following characterization of shadowing in subshifts over arbitrary $\Lambda$.

\begin{theorem}\label{shadowingSBT}
	Let $\Gamma$ be a subshift of $\Lambda^\omega$. Then $\Gamma$ is an SBT if and only if $(\Gamma,\sigma)$ has shadowing.
\end{theorem}

It is a standard result that subshifts over finite alphabet are \emph{positively expansive}, i.e. there exists an \emph{expansive constant} $c>0$ such that if $x,y\in\Gamma$ with $d(\sigma^n(x),\sigma^n(y))<c$ for all $n\in\omega$, then $x=y$. In fact, the expansive constant $c$ can be taken to be 1 in this setting, and this carries over to the setting of arbitrary $\Lambda$ perfectly.

\begin{theorem}
	Let $\Gamma$ be a subshift of $\Lambda^\omega$. Then $(\Gamma, \sigma)$ is positively expansive with expansive constant 1.
\end{theorem}

\begin{proof}
	
	Let $x,y\in\Gamma$ such that for all $n\in\omega$, $d(\sigma^n(x),\sigma^n(y))<1$. Then for all $n\in\omega$, we have $x_n=y_n$, and hence $x=y$.
	\end{proof}

As a consequence, we can show that shadowing implies asymptotic shadowing for subshifts of $\Lambda^\omega$. In fact, the following more general result holds.

\begin{theorem}\label{ShadImpliesAsymptShad}
If $(X,f)$ is a positively expansive surjective system with shadowing, then $(X,f)$ has asymptotic shadowing.
\end{theorem}

\begin{proof}
	
	Let $c$ be the expansive constant of $(X,f)$.

	Since $(X,f)$ has shadowing, choose $\delta_0>0$ so that every $\delta_0$-pseudo-orbit is $c/2$-shadowed.
	
	Let $\seq{x_i}$ be an asymptotic pseudo-orbit, and fix $N\in\omega$ such that $\seq{x_{i+N}}$ is a $\delta_0$-pseudo-orbit. Now, as $f$ is surjective, we choose a point $z\in X$ so that its $N$-th image under $f$, $f^N(z)$, $c/2$-shadows $\seq{x_{i+N}}$.
	
	In fact, this point necessarily asymptotically shadows $\seq{x_i}$. Indeed, fix $c/2>\epsilon>0$, and choose $\delta>0$ such that every $\delta$-pseudo-orbit is $\epsilon$-shadowed. As above, we can find a point $z'\in X$ and $N'$ such that $f^{N'}(z')$ $\epsilon$-shadows $x_{i+N'}$.
	
	In particular, for $n\geq \max\{N,N'\}$ we have
	
	\[d(f^n(z),f^n(z')\leq d(f^n(z), x_n)+d(x_n,f^n(z'))<c\]
	
	and so, $f^n(z)=f^n(z')$ as $f$ is positively expansive with constant $c$. In particular, for $n\geq \max\{N,N'\}$, $z$ \emph{eventually} $\epsilon$-shadows $\seq{x_i}$. As $\epsilon>0$ was arbitrary, it follows that $z$ asymptotically shadows $\seq{x_i}$.
	\end{proof}

\begin{corollary}\label{SBTAsymptotic}
	Let $\Gamma$ be a subshift of $\Lambda^\omega$. If $\Gamma$ is an SBT, then $(\Gamma,\sigma)$ has asymptotic shadowing.
\end{corollary}

Thus, in the context of shadowing in subshifts over infinite alphabet, subshifts of bounded type are the correct analogue for subshifts of finite type in the finite alphabet case.

\section{$\omega$-limit sets in subshifts of Baire space}

It should be noted however, that there are significant differences between these type of dynamical systems, in part due to the lack of compactness.

In a compact dynamical system, an $\omega$-limit set is necessarily internally chain transitive \cite{BGKR-omega}. In the context of shift spaces, the authors have demonstrated that there is a stronger connection between internal chain transitive sets and $\omega$-limit sets. Specifically, they have shown that in a subshift of finite type $\Gamma$, a closed, invariant subset $A$ is an $\omega$-limit set of some point $z\in\Gamma$ if and only if $A$ is internally chain transitive \cite{MR}. This connection proves to hold in other classes of dynamical systems exhibiting appropriate forms of shadowing \cite{GoodMeddaugh-ICT}.

Strikingly, in the context of shifts over infinite alphabets, even the first result fails.

\begin{remark} \label{counterexample}
	Consider the full shift with alphabet $\omega$. The point \[x=0^121^230^341^450^561^670^781^8\cdots\]
	has $\omega(x)=\{1^\omega,0^\omega\}$, which is closed and invariant, but not internally chain transitive.
\end{remark}

Perhaps more surprisingly, in subshifts of bounded type over \emph{countable} alphabets, the converse \emph{does} hold, i.e. closed invariant internally chain transitive sets are necessarily $\omega$-limit sets.

\begin{lemma}\label{ICTAsymptotic}
	Let $(X,f)$ be a dynamical system with $X$ separable metric. Let $Z\subseteq X$ be closed and internally chain transitive. Then there exists an asymptotic pseudo-orbit $\seq{x_i}$ in $Z$ such that $\omega\seq{x_i}=Z$.
\end{lemma}

\begin{proof}
	
	Since $X$ is separable metric, so is $Z$, and as such we can find a countable set $Z'\subseteq Z$ with $\overline{Z'}=Z$. Choose a sequence $\seq{z_n}$ in $Z'$ such that for each $N\in\omega$, $\{z_n:n\geq N\}=Z'$.
	
	Now, since $Z$ is internally chain transitive, for each $n\in\omega$, choose a finite $1/n$-pseudo-orbit $y^n_0, y^n_1,\ldots y^n_{k_n}$ in $Z$ with $y^n_0=z_n$ and $y^n_{k_n}=z_{n+1}$.
	
	We now define the asymptotic pseudo-orbit $\seq{x_i}$ as follows. For $i<k_0$, take $x_i=y^0_i$. For $i\in\omega$ with $\sum_{j=0}^n k_j\leq i<\sum_{j=0}^{n+1} k_j$, define $x_i=y^n_{i-\sum_{j=0}^n k_j}$, i.e.
	\[\seq{x_i}=y^0_0, y^0_1,\cdots, y^0_{k_0}=y^1_0, y^1_1,\cdots, y^1_{k_1}=y^2_0, y^2_1, \cdots y^2_{k_2}=y^3_0,\cdots. \]
	
	It is clear, then, that $\seq{x_i}$ is an asymptotic pseudo-orbit, and for each $N\in\omega$, we have $Z'\subseteq\{x_i:i\geq N\}\subseteq\overline{\{x_i:i\geq N\}}\subseteq Z$, and so $\omega\seq{x_i}=Z$.
	\end{proof}

It is worth noting that there are non-metric notions of shadowing and internal chain transitivity \cite{good-meddaugh}, and in these contexts the preceding lemma and proof can easily be modified to accommodate systems which are second countable.

We are now able to prove that in subshifts of bounded type, every closed ICT set is indeed an $\omega$-limit set. In fact, we are able to show that it is an $\omega$-limit set in a particularly tame sense.

\begin{theorem} \label{SBT-ICT}
	Let $\Gamma$ be a subshift of $\Lambda^\omega$ with $\Lambda$ countable. If $\Gamma$ is an SBT and $Z\subseteq \Gamma$ is a closed ICT set, then there exists $x\in\Gamma$ with $Z=\omega(x)$. Additionally, $x$ can be chosen such that for all $\epsilon>0$, there exists $N\in\omega$ such that $d(\sigma^i(x),Z)<\epsilon$ for all $i\geq N$.
\end{theorem}

\begin{proof}
	
	Let $\Lambda$ be countable and let $\Gamma$ be a subshift of bounded type.
	
	Let $Z\subseteq\Gamma$ be a closed ICT set. Since $\Lambda$ is countable, $\Lambda^\omega$ is separable metric and thus by Lemma \ref{ICTAsymptotic}, there exists an asymptotic pseudo-orbit $\seq{z_i}$ in $Z$ with $\omega\seq{z_i}=Z$. By Corollary \ref{SBTAsymptotic}, the system $(\Gamma,\sigma)$ has asymptotic shadowing, so we can find $x\in \Gamma$ which asymptotically shadows $\seq{z_i}$, and hence by Lemma \ref{omegaequalsomega}, $\omega(x)=\omega\seq{z_i}=Z$.	
	
	Now, fix $\epsilon>0$. Since $\seq{z_i}$ is contained in $Z$, and is asymptotically shadowed by $\seq{\sigma^i(x)}$, there exists $N\in\omega$ such that for $i\geq N$, we have $d(z_i,\sigma^i(x))<\epsilon$. In particular, we have $d(\sigma^i(x),Z)<\epsilon$ for all $i\geq N$.
%
%
\end{proof}

With this in mind, we say that a subset $Z$ of $\Lambda^\omega$ is an \emph{attracting $\omega$-limit set} if there is a point $x\in\Lambda^\omega$ such that $\omega(x)=Z$ and for all $\epsilon>0$, there exists $N\in\omega$ with $d(\sigma^i(x),Z)<\epsilon$ for all $i\geq N$.

\begin{corollary}
	Let $\Gamma$ be a subshift of $\Lambda^\omega$ with $\Lambda$ countable. If $\Gamma$ is an SBT, then a set $Z\subseteq\Gamma$ is a closed ICT set if and only if it is an attracting $\omega$-limit set.
\end{corollary}

\begin{proof}
	By the previous Theorem, for each closed ICT set $Z$, there is necessarily such a point $x\in\Gamma$.
	
	Now, suppose that $Z$ is an attracting $\omega$-limit set and let $x\in\Gamma$ witness this. Fix $\delta>0$ and $a,b\in Z$. Since $\sigma$ is uniformly continuous, fix $\eta>0$ such that $\eta<\delta/2$ and so that if $d(x,y)<\eta$, then $d(\sigma(x),\sigma(y))<\delta/2$.
	
	Since $x$ witnesses that $Z$ is an attracting $\omega$-limit set, find $N$ such that for $i>n$, $d(\sigma^i(x),Z)<\eta$. Now, fix $m>N$ and $n>m$ so that $d(\sigma^m(x),a)<\eta)$ and $d(\sigma^n(x),b)<\eta$. We now construct our $\delta$-chain from $a$ to $b$ in $Z$ by choosing $x_0=a$, $x_{n-m}=b$ and $x_i\in B_\eta(\sigma^{m+i}(x))\cap Z$ for $0<i<n-m$. By choice of $\eta$, for $i<n-m$, we have $d(\sigma^{m+i+1}(x),\sigma(x_i))<\delta/2$, and hence
	\[d(x_{i+1},\sigma(x_i))\leq d(x_{i+1},\sigma^{m+i+1}(x))+d(\sigma^{m+i+1}(x),\sigma(x_i))<\eta+\delta/2<\delta.\]
	Thus $x_0,x_1,\ldots x_{n-m}$ is a $\delta$-chain in $Z$ from $a$ to $b$, and thus $Z$ is ICT. As it is an $\omega$-limit set, it is also closed as desired.
\end{proof}

Having established this result for subshifts of bounded type, we turn our attention to subshifts of \emph{finite} type.

\begin{lemma} \label{FullshiftICT}
	Let $\Gamma$ be a subshift of $\Lambda^\omega$, with $\Lambda$ infinite (countably or uncountably so). If $\Gamma$ is an SFT, then $\Gamma$ is ICT.
\end{lemma}

\begin{proof}
	Let $\Lambda$ be an infinite set and $\Gamma$ a subshift of $\Lambda^\omega$ which is an SFT. Let $\mathcal F(\Gamma)$ be a finite basis for the forbidden words of $\Gamma$, and let $c\in\Lambda$ such that $c$ does not appear as a symbol in any word in $\mathcal F(\Gamma)$.

	Let $x=\seq{x_i}$ and $y=\seq{y_i}$ belong to $\Gamma$, and let $\delta>0$. Additionally fix $N\in\omega$ such that if $a,b\in \Gamma$ and $a_{[0,N)}=b_{[0,N)}$, then $d(a,b)<\delta$.
	
	Define $z=x_{[1,N+1)}cy\in\Lambda^\omega$, notice that each subword of $z$ is either a subword of $x$, of $y$, or contains the symbol $c$, and as such does not belong to $\mathcal F(\Gamma)$, and so $z\in\Gamma$.
	
	Furthermore,  notice that \[x, z, \sigma(z),\ldots \sigma^{N}(z)=cy,\sigma^{N+1}(z)=y\] is a $\delta$-chain from $x$ to $y$ in $\Gamma$. As $x,y\in\Gamma$ and $\delta>0$ were arbitrary, we see that $\Gamma$ is ICT.
\end{proof}

This result, combined with Theorem \ref{SBT-ICT} yields the following.

\begin{corollary} \label{countableSFT}
	Let $\Gamma$ be a subshift of $\Lambda^\omega$ with $\Lambda$ countably infinite. If $\Gamma$ is an SFT, then there exists $x\in\Gamma$ with  $\omega(x)=\Gamma$.
\end{corollary}

In fact, this is a property unique to countably infinite alphabets.

\begin{corollary} \label{dichotomy}
	Let $\Lambda$ be set containing more than one element. Then $\Lambda$ is countably infinite if and only if every subshift $\Gamma$ of $\Lambda^\omega$ which is an SFT has a point $z\in\Gamma$ with $\omega(z)=\Gamma$.
\end{corollary}

\begin{proof}
	By Corollary \ref{countableSFT}, if $\Lambda$ is countably infinite, then the property that every subshift of finite type is an $\omega$-limit set holds.
	
	If $\Lambda$ is finite and contains the symbols $s_0,s_1$, consider the finite set $\mathcal F=\{s_0s_1,\}\cup\left(\Lambda\setminus\{s_0,s_1\}\right)$, and let $\Gamma$ be the subshift of finite type with $\mathcal F$ as a basis for its forbidden words. Clearly $\Gamma$ is an SFT; but it is not ICT, as there can be no $\delta$-chain from $s_0^\infty$ to $s_1^\infty$ for $\delta<1$. As such, by the results of \cite{BGKR-omega}, it is not an $\omega$-limit set.
	
	If $\Lambda$ is uncountable, then the full shift $\Lambda^\omega$ is an SFT, and by Lemma \ref{FullshiftICT} is ICT as well. However, if $x=\seq{x_i}\in\Lambda^\omega$, then $\omega(x)\subseteq\{x_i:i\in\omega\}^\omega\subsetneq\Lambda^\omega$, and thus $\Lambda^\omega$ is not an $\omega$-limit set.
\end{proof}

It is worth noting that there exist subshifts $\Gamma$ of bounded type over countable alphabets which are not internally chain transitive, and indeed, for which $\Gamma$ is not an $\omega$-limit set.

\begin{remark} \label{SBTnotOmega}
	Consider the subshift $\Gamma$ of $\omega^\omega$ with basis of forbidden words given by $\mathcal F(\Gamma)=\{st:s>t\}$.
	
	For $\delta<1$, let $x_0,x_1,\ldots x_n$ be a $\delta$-chain in $\Gamma$. Let $a\in\omega$ be the initial symbol of $x_0$. Then for all $i\leq n$, if $b$ is a symbol of $x_i$, then $b\leq a$. In particular, there is no $\delta$-chain from $x_0$ to the point $(a+1)^\infty$.
	
	To see that $\Gamma$ is not an $\omega$-limit set, observe that for $x\in \Gamma$ with initial symbol $a$, the forward orbit $\{\sigma^n(x)\}$ of $x$ is contained in $\{0,1,\ldots a\}^\omega\subsetneq\Gamma$.
\end{remark}

It is interesting to note that the example from Remark \ref{counterexample}, while not internally chain transitive, is in fact, a union of closed internally chain transitive sets. While this might lead one to conjecture that in shifts over infinite alphabets, every $\omega$-limit set is a union of closed internally chain transitive sets, we observe that this is false, as the following example demonstrates.

\begin{remark}
	Consider the full shift with alphabet $\omega$. The point \[x=0120^21^230^31^340^41^450^51^560^61^670^71^780^81^8\cdots\]
	has $\omega(x)=\{1^\omega,0^\omega\}\cup\{0^k1^\omega: k\in\omega\}$, which is closed and invariant, but is not the union of closed internally chain transitive sets.
\end{remark}

This example demonstrates how significantly the dynamics of shift spaces differ between the finite and infinite alphabet cases. In fact, for subshifts of finite type over infinite alphabet, \emph{every} closed invariant set is an $\omega$-limit set. Note that this is false for subshifts of bounded type, as witnessed in Remark \ref{SBTnotOmega}.

\begin{theorem}
	Let $\Gamma$ be a subshift of $\Lambda^\omega$ with $\Lambda$ countable. If $\Gamma$ is an SFT, then $Z\subseteq \Gamma$ is a closed, invariant set if and only if it is an $\omega$-limit set.
\end{theorem}

\begin{proof}	
	Let $\Lambda$ be an infinite set and $\Gamma$ a subshift of $\Lambda^\omega$ which is an SFT. Let $\mathcal F(\Gamma)$ be a finite basis for the forbidden words of $\Gamma$, and let $\{s_i:i\in\omega\}$ be a countably infinite subset of $\Lambda$ such that for each $i$, the symbol $s_i$ does not appear in any word of $\mathcal F(\Gamma)$.
	
	Now, let $Z$ be a closed, invariant subset of $\Gamma$, and let $B=\seq{b_i}_{i\in\omega}$ be an enumeration of the initial segments of points in $Z$, i.e. for each $w\in Z$ and $n>0$, there exists $i\in\omega$ such that the word $w_{[0,n)}$ is equal to $b_i$. Note that, since $Z$ is invariant, for any $w\in Z$, \emph{every} subword of $w$ is an initial segment of a point in $Z$, and as such $B$ is also an enumeration of the subwords of points in $Z$. It follows that if $b$ is a subword of $b_i$ for some $i\in\omega$, then there exists $j\in\omega$ with $b=b_j$.
	
	Define a point $x\in\Lambda^\omega$ as follows.
	\[x=b_0s_0b_1s_1\cdots b_is_i\cdots\]
	We claim that $x\in\Gamma$ and $\omega(x)=Z$. For the former, let $w$ be a subword of $x$. Then $w$ is either a subword of $b_i$ for some $i\in\omega$ or else contains $s_j$ for some $j\in\omega$. In either case, $w\in\mathcal F(\Gamma)$, and so $x\in\Gamma$.
	
	To see that $\omega(x)=Z$, let $z\in Z$. For each $m>0$, there exists $i(m)\in\omega$ with $z_{[0,m)}=b_{i(m)}$. In particular, for each $n>0$, the initial segment $z_{[0,n)}$ of $z$ is a subword of $b_{i(m)}$ for all $m\geq n$, and as such occurs in $x$ infinitely often. It follows that $z\in \omega(x)$.
	
	Now, consider $p\in\omega(x)$, and let $n>0$. Then $p_{[0,n)}$ occurs in $x$ infinitely often. As $p_{[0,n)}$ is finite, there is $N\in\omega$ such that it does not contain the symbol $s_i$ for any $i\geq N$. In particular, $p_{[0,n)}$ occurs infinitely often in
	\[b_Ns_Nb_{N+1}s_{N+1}\cdots b_{N+i}s_{N+i}\cdots\]
	and, since it does not contain any of the symbols $s_{N+i}$ for $i\geq 0$, we see that $p_{[0,n)}$ is a subword of $b_{N+i}$ for infinitely many $i$, and in particular, for at least one $i$. As such, $p_{[0,n)}$ is the initial segment of some point in $Z$. Thus, every neighborhood of $p$ meets $Z$, and since $Z$ is closed, $p\in Z$. Thus $Z=\omega(x)$.
	
	For the converse, observe that if $Z$ is an $\omega$-limit set, it is necessarily closed and invariant.
\end{proof}

Thus, in the context of subshifts of finite type with countable alphabet, the $\omega$-limit sets and the closed invariant sets coincide. Since subshifts of finite type are also subshifts of bounded type, we also see that the attracting $\omega$-limit sets and the closed ICT sets coincide.

\bibliographystyle{plain}
\bibliography{ComprehensiveBib}

%
%

\end{document}